\documentclass[10pt]{article}
 \usepackage{amsmath, amsfonts, amsthm, amssymb, amscd, enumerate, url, slashed, stmaryrd, upgreek, thmtools, mathdots}
 \usepackage{float}
 \usepackage{tikz}
 \usetikzlibrary{matrix}
 
\begingroup
 \makeatletter
 \@for\theoremstyle:=definition,remark,plain\do{%
 \expandafter\g@addto@macro\csname th@\theoremstyle\endcsname{%
 \addtolength\thm@preskip\parskip
 }%
 }
 \endgroup
 
 \usepackage[titletoc,toc,title]{appendix}

 \usepackage[colorlinks = true,
 linkcolor = blue,
 urlcolor = cyan,
 citecolor = red,
 anchorcolor = blue,
 pagebackref]{hyperref}
 
\usepackage[alphabetic,backrefs]{amsrefs}
 
 \usepackage[parfill]{parskip}
 \usepackage{anysize}
 \marginsize{1in}{1in}{1in}{1in}

\renewcommand{\b}{\beta}
\renewcommand{\d}{\delta}
\newcommand{\g}{\gamma}
\newcommand{\e}{\varepsilon}

\newcommand{\real}{\operatorname{\mathrm{Re}}}
\newcommand{\imag}{\operatorname{\mathrm{Im}}}

\newcommand{\Spin}[1]{\mathrm{Spin}(#1)}

\newcommand{\SU}[1]{\mathrm{SU}(#1)}

\newcommand{\R}{\mathbb R}
\newcommand{\C}{\mathbb C}

\newcommand{\Z}{\mathbb Z}

\newcommand{\Gtwo}{\mathrm{G}_2}

\newcommand{\la}{\langle}
\newcommand{\ra}{\rangle}

\newcommand{\ZZ}{\mathbb{Z}}

\newcommand{\Fix}{\operatorname{Fix}}
\newcommand{\diag}{\operatorname{diag}}

\newcommand{\CC}{\mathbb{C}}
\newcommand{\CP}{\mathbb{CP}}

\newcommand{\hook}{\lrcorner\,}
\newcommand{\re}{\mathrm{Re}}

\renewcommand{\b}{\beta}

\theoremstyle{plain}
\newtheorem{proposition}{Proposition}

\newtheorem{theorem}[proposition]{Theorem}

\newtheorem{lem}[proposition]{Lemma}
\theoremstyle{definition}

\theoremstyle{remark}
\newtheorem{remark}{Remark}

\begin{document}

\title{Compact holonomy $\Gtwo$ manifolds need not be formal}

\author{ Luc\'ia Mart\'in-Merch\'an, \\
\emph{Department of Pure Mathematics, University of Waterloo} \\ \tt{lucia.martinmerchan@uwaterloo.ca} }

\maketitle

\begin{abstract}
\noindent We construct a compact, simply connected manifold with holonomy $\Gtwo$ that is non-formal. We use the construction method of compact torsion-free $\Gtwo$ manifolds developed by D.D. Joyce and S. Karigiannis. A non-vanishing triple Massey product is obtained by arranging the singular locus in a  particular configuration.
\end{abstract}

\tableofcontents

\section{Introduction} \label{sec:introduction}

The link between special holonomy and formality was discovered by P. Deligne, P. Griffiths, J. Morgan, and D. Sullivan in \cite{DGMS}. They proved that compact manifolds satisfying the $dd^c$-Lemma are formal. In particular, Kähler, Calabi-Yau, and hyper-Kähler compact manifolds are formal. Later, M. Amann and V. Kapovitch showed in \cite{AK} that compact quaternionic-Kähler manifolds with positive scalar curvature are also formal. For that, they use that their twistor space admits a Kähler metric, along with results from rational homotopy.

The question of formality for compact manifolds with exceptional holonomy
has remained open
since they were first constructed by D. D. Joyce in \cites{J1,J2,J3,J4}. In fact, these were conjectured to be formal in \cite[Conjecture 2]{AK}.
Various approaches have been explored to better understand the $\Gtwo$ case. Analytical methods were employed in \cites{CKT, Verbitsky}, but these seem to be inconclusive on their own. More precisely, in \cite{Verbitsky} M. Verbitsky extended the Kähler identities to Riemannian manifolds equipped with a parallel form, and built a cohomology algebra with special properties using them. From this, he provided an alternative proof of formality for Kähler manifolds. This idea was later applied in \cite{CKT} to torsion-free $\Gtwo$ manifolds, where they show that these are almost formal—meaning that all triple Massey products vanish except possibly those involving three degree-$2$ cohomology classes.

Another strategy is contained in \cite{CN}, where they studied the rational homotopy type of simply connected $7$-manifolds, among others. In particular, they built the Bianchi-Massey tensor as a tool to detect non-formality of $7$-manifolds, and they 
concluded that if there was a non-formal simply connected manifold with holonomy $\Gtwo$, it would have $b_2 \geq 4$.  Prior to \cite{CN}, G.R. Cavalcanti proved a weaker version of this result in \cite{Cavalcanti} and he also obtained a compact non-formal $7$-dimensional manifold that satisfies the known topological obstructions for holonomy $\Gtwo$ manifolds. 

The
 relative scarcity of holonomy $\Gtwo$ manifolds makes it hard to determine what phenomena might influence whether they are formal or not. Formality of one of Joyce's examples was proved in \cite{AT}. The examples by J. Nordström in \cite{Nord} are formal because they have $b_2 \leq 1$. 
 The author is unaware of any attempts to address the formality of the examples with
$b_2\geq 4$ provided by A. Kovalev \cite{Kovalev}, or by A. Kovalev and N. H. Lee \cite{KL}, or by  A. Corti, M. Haskins, J. Nordström and T. Pacini \cite{CHNP}.
In this paper, we use the desingularization method by D. D. Joyce and S. Karigiannis \cite{JK} to provide a negative answer to the question of formality for compact manifolds with holonomy $\Gtwo$. We prove the following.

\begin{theorem}
There exists a compact simply-connected manifold with holonomy $\Gtwo$ that is non-formal.
\end{theorem}

The constructed manifold $(\widetilde{M}, \widetilde{\varphi})$ is obtained by resolving an orbifold $(X,\varphi)$, which itself is the quotient of a flat $7$-manifold under the action of an involution with fixed points. The singular locus of $X$ consists of $10$ disjoint associative manifolds, $N_1, \dots, N_{10}$, each of which is $2:1$ covered by a torus. This orbifold can also be resolved employing Joyce's method in \cite{J2}.
It was known that the cohomology groups of $\widetilde{M}$ are determined by those of $X$ and $N_1, \dots, N_{10}$ (see equation \eqref{eqn:cohom}), and the product structure depends on the embedding of the singular locus of $X$. 
More precisely, the square of the Thom class $\mathbf{x}_j$ of a connected component $E_j$ in the exceptional divisor lies in $H^*(X)$, and it depends on the Thom class of the corresponding connected component $N_j$ from which it arises (see expression \eqref{eqn:prod-1}). Here, the orientation of $N_j$ is determined by $\varphi$.

The novelty of this example is that we obtain a non-vanishing triple Massey product, even when the orbifold and the singular locus are formal, by positioning four of the connected components in a special configuration. Aside from the cohomology product formulas, the facts that guarantee that the triple Massey product
$
\la \mathbf{x}_1 + \mathbf{x}_2, \mathbf{x}_7 + \mathbf{x}_3, \mathbf{x}_7 - \mathbf{x}_3\ra
$
 is well-defined and non-vanishing are the following:
\begin{enumerate}
\item There are oriented cobordisms between $N_1$ and $N_2$ and between $N_3$ and $N_7$.
\item The cobordism between $N_3$ and $N_7$ intersects $N_1$ negatively and $N_2$ positively.
\end{enumerate}

This contrasts with the example examined in \cite{AT}, where both the orbifold and the singular locus are also formal, but the cobordisms between singular components can be chosen to avoid transverse intersections with other components. This also differs from the examples of compact manifolds with $b_1=1$ that are equipped with a closed $\Gtwo$ structure and obtained by orbifold resolution techniques. The formal manifold in \cite{FFKM} satisfies the same properties as the example analyzed in \cite{AT}. The manifold in \cite{LMM} is non-formal because the resolved orbifold is non-formal. 

This paper is organized as follows. In section \ref{sec:manifold} we construct $\widetilde{M}$ and we prove that it is simply connected and it admits a metric with holonomy $\Gtwo$. In section \ref{sec:nonformality} we obtain its cohomology algebra and show it is non-formal. The author has chosen to present most of the computations explicitly to ensure that the discussion is self-contained.

\textbf{Acknowledgements.} I would like to thank Spiro Karigiannis for carefully proofreading this paper and for providing suggestions that improved its exposition. I am also grateful to Dominic Joyce for his remarks, and Vicente Muñoz for his interest and opinion.

 \newpage

\section{A simply connected compact manifold with holonomy $\Gtwo$} \label{sec:manifold}

Consider the flat $6$-torus $T^6 = (\CC/\Gamma)^3$, where $\Gamma=\Z \la 1, i \ra$, endowed with its standard flat $\SU{3}$ structure $(\omega,\Theta)=(\frac{i}{2}\sum_k dz_{k\bar{k}}, dz_{123})$. The isometry
\begin{equation} \label{eqn:mapping-function}
f\colon T^6 \to T^6, \qquad  f[z_1,z_2,z_3]=[iz_1,iz_2,-z_3]
\end{equation}
preserves the $\SU{3}$ structure. 
The manifold
\begin{equation} \label{eqn:M-def}
M=\R \times T^6/(t,[z])\sim (t+k,f^k[z]), \quad k\in \Z,
\end{equation}
equipped with the flat metric
has a parallel $\Gtwo$ structure 
$
\varphi=dt\wedge \omega+ \re(\Theta)
$. Since $f^4=\mathrm{Id}$, the $7$-torus  
$
\mathbb{T}= \R/4\ZZ \times T^6
$ covers $M$.
The Deck transformation group of the covering is $\Z_4$, generated by $F[t,[z]]=[t+1,f[z]]$. We quotient $M$ by the involutions 
\begin{equation}\label{eqn:iota}
\iota_1,\iota_2 \colon M \to M, \qquad \iota_1[t,[z]]=[t,[-z_1,-z_2,z_3+1/2]], \qquad 
\iota_2[t,[z]]=[t,[-z_1,-z_2,z_3+i/2]].
\end{equation}
Neither of these involutions, nor their compositions, have fixed points. Since they preserve $\varphi$, the quotient space
$M'=M/\la \iota_1,\iota_2 \ra$ is a smooth manifold with a parallel $\Gtwo$ structure and a flat metric. Observe that there is a fibration $T^6/\Z_2^2 \hookrightarrow M' \to S^1$. In addition, the involution 
\begin{equation}\label{eqn:kappa}
\kappa \colon M \to M, \qquad \kappa[t,[z]]=[1-t,\bar{z}_1,\bar{z}_2,\bar{z}_3],
\end{equation}
is well-defined because
$
\kappa[t+1,f[z]]=[-t,-i\bar{z}_1,-i\bar{z}_2,-\bar{z}_3]=[1-t,\bar{z}_1,\bar{z}_2,\bar{z}_3]=\kappa[t,[z]].
$
It determines a map $\kappa \colon M' \to M'$ as
it commutes with both $\iota_1$ and $\iota_2$. Moreover, $\kappa$ has fixed points and $\kappa^*(\varphi)=\varphi$.  %This map $\kappa$ fixes $\varphi$ as it preserves both $dt\wedge \omega$ and $\re(\Theta)$.

The orbifold
$X=M'/\kappa$ can be alternatively described as the quotient of the torus $\mathbb{T}=\R/4\Z \times T^6$ by the group $D_4\times \Z_2^2$. Here $D_4$ is the dihedral group spanned by $\kappa$ and $F$ (note that $\kappa \circ F=F^{-1}\circ \kappa$), whereas $\Z_2^2$ is generated by $\iota_1$ and $\iota_2$.

\begin{remark} 
The orbifold $X$ is related to the examples provided by Joyce in \cite{J2}. More precisely, a change of variables transforms $X_1=\mathbb{T}^7/\la F, \kappa \ra$ into Example $9$, and $X_2=\mathbb{T}^7/\la F, \kappa, \iota_1 \circ \iota_2 \ra$ into Example $10$. In particular,
there is a branched covering $X_2 \to X$, because 
$\iota_1$ has fixed points when it acts on $X_2$ (a fixed point of $\iota_1$ on $X_2$ is the projection of $p=[0,0,0,1/4]$ as $\iota_1 (\kappa (p)) =[1,0,0,3/4]=p$).
\end{remark}

\subsection{Singular locus} \label{sec:sing}
The singular locus of $X$ is the projection of the fixed locus of $\kappa$, $\kappa \circ \iota_1$, $\kappa \circ \iota_2$, and $\kappa \circ \iota_1\circ \iota_2$ on $M$.  We observe that if $\gamma$ is one of these maps, and we denote it by $\gamma[t,[z]]=[1-t,\gamma' [z]]$, then a point $[t,[z]]$ with $t\in [0,1)$ is fixed by $\gamma$ if and only if $t=1/2$ and $[z]\in \Fix(\gamma')$, or $t=0$ and $[z]\in \Fix(f^{-1}\circ \gamma')$.
We now compute these and analyze their projections to $X$. 

We begin with the level set $t=0$. From $f^{-1}[z]=[-iz_1,-iz_2,-z_3]$, we obtain
\begin{align*}
f^{-1}\circ \kappa' [z]=&[-i\bar{z}_1,-i\bar{z}_2,-\bar{z}_3], \qquad \qquad 
&f^{-1}\circ (\kappa\circ \iota_1)' [z]=[i\bar{z}_1,i\bar{z}_2,-\bar{z}_3 + 1/2], \\
f^{-1}\circ (\kappa\circ \iota_2)' [z]=&[i\bar{z}_1,i\bar{z}_2,-\bar{z}_3 + i/2], \qquad 
&f^{-1}\circ (\kappa\circ \iota_1 \circ \iota_2)' [z]= [-i\bar{z}_1,-i\bar{z}_2,-\bar{z}_3 + (1+i)/2].
\end{align*}
Hence, $\Fix(f^{-1}\circ (\kappa \circ \iota_2)')=\Fix(f^{-1}\circ (\kappa \circ \iota_1\circ \iota_2)')=\emptyset$ (these act on $y_3$ as $y_3 \to y_3 + 1/2$) and,
\begin{align}
\Fix(f^{-1}\circ \kappa')=&\{ [x_1-ix_1, x_2-ix_2, \e_3 + iy_3],\quad \e_3 \in \{0,1/2\}, \,\, x_1,x_2,y_3\in [-1/2,1/2] \} , \label{eqn:N1} \\
\Fix(f^{-1}\circ (\kappa\circ \iota_1)')=&\{ [x_1+ix_1, x_2+ix_2, \delta_3 + iy_3], \quad \delta_3 \in \{-1/4,1/4\}, \,\,  x_1,x_2,y_3\in [-1/2,1/2] \}. \label{eqn:N2}
\end{align}
We denote $N_1^1, N_1^2 \subset M$ the  connected components of $\Fix(\kappa)$ at $t=0$ determined by $\e_3=0$ and $\e_3=1/2$, respectively. Similarly, we set $N_2^1,N_2^2\subset M$ the components of $\Fix(\kappa \circ \iota_1)$ at $t=0$ determined by $\delta_3=-1/4$ and $\delta_3=1/4$. Each of these is diffeomorphic to a $3$-torus. The map $\iota_1$
interchanges $N_j^1$ with $N_j^2$, while $\iota_2$
preserves each connected component.
In addition, $N_j^k \to N_j^k/\iota_2$ is a $2:1$ covering, and $H^1(N_j^k/\iota_2)=\la dy_3\ra$. The form $dy_3$ is harmonic because $N_j^k$ is equipped with the flat metric.
In fact, the map $N_1^1/\iota_2 \to \R/\Z$, $[x_1-ix_1,x_2-ix_2,iy_3]\to [2y_3]$ is a submersion with a $T^2$ fiber. A similar statement holds for the remaining connected components. Therefore, the fixed locus of $\kappa$ on $M'$ at $t=0$ has two connected components, $N_1'$ and $N_2'$. Both admit the nowhere-vanishing  harmonic $1$-form $dy_3$. We denote by $N_1$ and $N_2$ their projections to $X$.

To analyze the singularities at $t=1/2$, we compute
\begin{align*}
\kappa' [z]=&[\bar{z}_1,\bar{z}_2,\bar{z}_3], \qquad \qquad 
&(\kappa\circ \iota_1)' [z]=[-\bar{z}_1,-\bar{z}_2,\bar{z}_3 + 1/2], \\
(\kappa\circ \iota_2)' [z]=&[-\bar{z}_1,-\bar{z}_2,\bar{z}_3 + i/2], \qquad 
&(\kappa\circ \iota_1 \circ \iota_2)' [z]= [\bar{z}_1,\bar{z}_2,\bar{z}_3 + (1+i)/2].
\end{align*}
Hence, $\Fix((\kappa \circ \iota_1)')=\Fix((\kappa \circ \iota_1\circ \iota_2)')=\emptyset$ and,
\begin{align}
\Fix(\kappa')=&\{ [x_1+ i\e_1, x_2+ i\e_2, x_3 + i\e_3 ], \quad \e_1,\e_2, \e_3 \in \{0,1/2\}, \,\,  x_1,x_2,x_3\in [-1/2,1/2] \} , \label{eqn:N3-6}\\
\Fix((\kappa\circ \iota_2)')=&\{ [\e_1+ iy_1, \e_2+iy_2, x_3 + i\delta_3], \quad \e_1,\e_2, \d_3-1/4 \in \{0,1/2\}, \,\,  y_1,y_2,x_3\in [-1/2,1/2] \}.\label{eqn:N7-10}
\end{align}
We denote by $N_3^1,N_3^2 \subset M$ the connected components of $\Fix(\kappa)$ at $t=1/2$ determined by $\e_1=\e_2=0$ and $\e_3=0$ or $\e_3=1/2$ respectively. We use lexicographic order for the pairs 
$(\e_1,\e_2)$ and label the remaining connected components of $\Fix(\kappa)$ at $t=1/2$ by $N_4^1,N_4^2,N_5^1,N_5^2,N_6^1,N_6^2$,
with the upper index $1$ corresponding to 
$\e_3=0$. Similarly, we denote the connected components of $\Fix(\kappa \circ \iota_1)$ by $N_j^k$, where $j=7,\dots,10$, $k=1,2$, and the upper index $k=1$ now indicates $\delta_3=-1/4$. All these components are diffeomorphic to a $3$-torus.
We observe that $\iota_1$  preserves each connected component, while $\iota_2$ 
swaps $N_j^1$ with $N_j^2$. Also, $N_j^k \to N_j^k/ \iota_1 $ is a $2:1$ covering, and $H^1(N_j^k/\iota_1)=\la dx_3\ra$. Again, $dx_3$ is indeed harmonic on $N_j^k$ and there are submersions $N_j^k/\iota_1 \to \R/\Z$ with a $T^2$ fiber.
The fixed locus of $\kappa$ on $M'$ at $t=1/2$ has eight connected components, $N_3',\dots, N_{10}'$. These admit the nowhere-vanishing harmonic $1$-form $dx_3$. We denote by $N_3,\dots, N_{10}$ their projections to $X$.

 The submanifolds $N_j^k$ and $N_j'$ are associative submanifolds of $(M,\varphi)$ and $(M',\varphi)$ (see \cite[Proposition 2.13]{JK}).
\emph { We will always assume that $N_j^k$ (and $N_j'$, $N_j$) are oriented by $\varphi|_{N_j^k}=\re(dz_{123})|_{N_j^k}$ (and $\varphi|_{N_j'}$, $\varphi|_{N_j}$).}
Using the coordinates at $t=0$ and $t=1/2$ provided by equations \eqref{eqn:N1},\eqref{eqn:N2},\eqref{eqn:N3-6} and \eqref{eqn:N7-10} these forms are:
\begin{align}
\varphi|_{N_1^k}=&\, 2\, dx_1\wedge dx_2 \wedge dy_3, \qquad k=1,2, \label{eqn:or-1} \\
\varphi|_{N_2^k}=& -2\, dx_1\wedge dx_2 \wedge dy_3, \qquad k=1,2, \label{eqn:or-2} \\
\varphi|_{N_j^k}=&\,dx_1\wedge dx_2 \wedge dx_3, \qquad j=3,4,5,6 \quad k=1,2, \label{eqn:or-3-6} \\
\varphi|_{N_j^k}=&-dy_1\wedge dy_2 \wedge dx_3, \qquad j=7,8,9,10, \quad k=1,2. \label{eqn:or-7-10}
\end{align}
 Note that if $j=1$ (or $j=2$), the projections of $N_j^k$ to the first and second factor of $(\C/\Gamma)^3$ are the loops $[x-ix]$ (or $[x+ix]$), $x\in [-1/2,1/2]$. These have length $\sqrt{2}$. The projection to the third factor has, of course, length $1$. Therefore, the volume of these tori is $2$, which is why the factor of $2$ appears.

\subsection{Resolution} \label{sec:resolution}

The resolution method developed in \cite{JK} allows to desingularize the orbifold X and to obtain a metric
with holonomy $\Gtwo$ on the resolution. 

\begin{theorem}
The resolution $\widetilde{M}$ of $X$ is a compact simply connected manifold that 
admits a metric with holonomy $\Gtwo$.
\end{theorem}
\begin{proof}
Since each connected component of the singular locus of $X$ admits a nowhere-vanishing harmonic $1$-form, the resolution method  \cite[Theorem 1.1]{JK}
ensures that there
is a resolution $\rho \colon \widetilde{M} \to X$ and a torsion-free $\Gtwo$ structure $\widetilde{\varphi}$ on $\widetilde{M}$. It satisfies $\pi_1(\widetilde{M})=\pi_1(X)$. According to \cite[Proposition 10.2.2]{Joyce2}, the holonomy of $(\widetilde{M},\widetilde{\varphi})$ equals $\Gtwo$ if and only if $\pi_1(\widetilde{M})$ is finite. To finish the proof, it suffices to show that $X$ is simply connected.

We first compute $\pi_1(M,p_0)$ with $p_0=[1/2,[0,0,0]]$. From the long exact sequence of the fibration $T^6 \hookrightarrow M \to \R/\Z$, $[t,[z]]\longmapsto [t]$, we obtain the short exact sequence
$$
\{1\} =\pi_2(S^1) \to \pi_1(T^6) \to \pi_1(M) \to \pi_1(S^1) \to \{1\}.
$$
Of course, the loop ${\g}_0(s)= [1/2+s,[0,0,0]]$ in $M$  maps to the generator of $\pi_1(S^1)$, and provides a splitting of the short exact sequence. Therefore, $\pi_1(M)=\pi_1(S^1) \ltimes \pi_1(T^6)$. We now analyze the product of $[\gamma_0]$ with elements in the image of $\pi_1(T^6) \hookrightarrow \pi_1(M)$. Given a loop $\gamma(s)=[1/2,\widehat{\gamma}(s)]$ on $M$ with $\widehat{\gamma}(0)=[0,0,0]$, the loop $\g_0^{-1} \cdot \gamma \cdot \g_0$ is homotopic to 
$s \to [3/2,\widehat{\gamma}(s)]=[1/2,f^{-1} \circ \widehat{\gamma}(s)]$ on $M$. (As usual, we denote $\gamma^{-1}(s)=\gamma(1-s)$.)
The homotopy $H_{r}(s)$ is a sequence of three paths. First, move along the $t$-direction from $[3/2,[0,0,0]]$ to $[3/2-r,[0,0,0]]$. Then travel on the level set $3/2-r$ via the map $s\to  [3/2-r,\widehat{\gamma}(s)]$, and finally return to $[3/2,[0,0,0]]$ by reversing the initial path. 

We denote by $\g_k(s)=[1/2,\widehat{\gamma}_k(s)]$, $k=1,\dots,6$ the canonical generators of $\pi_1(\{1/2\}\times T^6,p_0)$. Namely, $\widehat{\gamma}_1(s)=[s,0,0]$, $\dots$, $\widehat{\gamma}_6(s)=[0,0,is]$. For $k=1,2$ we have
$f^{-1}\circ \widehat{\gamma}_{2k-1}=\widehat{\gamma}_{2k}^{-1}$ and $f^{-1}\circ \widehat{\gamma}_{2k}=\widehat{\gamma}_{2k-1}$. In addition, if $k=5,6$ then $f^{-1}\circ \widehat{\gamma}_{k}=\widehat{\gamma}_k^{-1}$. This implies the following relations:
\begin{align}
 \label{eqn:rel-fund-1}
[{\g}_{2k}]^{-1}&=[{\g}_0]^{-1}[ {\g}_{2k-1}][{\g}_0],\qquad
&[{\g}_{2k-1}]=[{\g}_0]^{-1}[{\g}_{2k}][{\g}_0], &\qquad k=1,2, \\
\label{eqn:rel-fund-2}
[{\g}_{5}]^{-1}&=[{\g}_0]^{-1}[ {\g}_{5}] [{\g}_0],\qquad
&[{\g}_{6}]^{-1}=[{\g}_0]^{-1}[{\g}_{6}][{\g}_0].\quad &
\end{align}

The projection $q\colon M\to M'$ is a $4:1$ cover with Deck transformation group $\la \iota_1,\iota_2\ra = \Z_2^2$.  There is a short exact sequence,
$$
\{1\} \to \pi_1(M) \to  \pi_1(M') \to \la \iota_1,\iota_2 \ra \to \{1\}.
$$

Consider the paths 
$$\widetilde{\gamma}_{5}, \widetilde{\gamma}_{6}\colon [0,1]\to M, \qquad 
\widetilde{\gamma}_{5}(s)=[1/2,[0,0,s/2]]
,\qquad
\widetilde{\gamma}_{6}(s)=[1/2,[0,0,is/2]].
$$
Then $\gamma_5'=q \circ \widetilde{\gamma}_{5}$ and $\gamma_6'=q \circ \widetilde{\gamma}_{6}$ are loops on 
 $M'$. In the long exact sequence, their homotopy classes project onto $\iota_1$ and $\iota_2$ respectively.  If $0\leq k\leq 4$ we denote  $\gamma_k'=q\circ \gamma_k$. In addition, define the paths on $M$
$$
\widetilde{\gamma}_{0}, {\gamma}_7 \colon [0,1]\to M, \qquad
\widetilde{\gamma}_0(s)=[1/2 +s,[0,0,1/2]], \qquad 
{\gamma}_7(s)=[3/2,[0,0,s/2]]=[1/2,0,0,-s/2].
$$
There is a base-point preserving homotopy on $M$ from $\gamma_0^{-1} \cdot \gamma_5' \cdot \widetilde{\gamma}_0$
to ${\gamma}_7$. The way it is constructed is the same as before.
Since $\gamma_0'=q\circ \widetilde{\gamma}_0$ and $q \circ {\gamma}_7 (s)=q[1/2,0,0,1/2-s/2]= (\gamma_5')^{-1}(s)$
we obtain $[ \gamma_0']^{-1}[{\gamma}_5'][ \gamma_0']= [\gamma_5']^{-1} $ on $M'$; a similar argument shows $[\gamma_0']^{-1}[{\gamma}_6'][ \gamma_0']= [\gamma_6']^{-1} $. 

The involution $\kappa \colon M' \to M'$ has fixed points; therefore the map $q'_*\colon \pi_1(M') \to \pi_1(X)$  induced by the projection $q'\colon M' \to X$ is surjective (see \cite[Chapter 2, Corollary 6.3]{Bre}). 
Observe that $\kappa$ reverses the direction of $\gamma_0'$, $\gamma_2'$, $\gamma_4'$ and ${\gamma}_6'$. That is, if $\gamma$ is one of these loops, then $\kappa \circ \gamma(s)=\gamma^{-1}(s)$.
This implies that $q_*'[\gamma]=1$ because $q'\circ \gamma= q'\circ \gamma''(s)$ where
$$
\gamma''(s)=
\begin{cases}
\gamma(s), &\mbox{if } t \leq 1/2, \\
\gamma^{-1}(s), &\mbox{if } t \geq 1/2,
\end{cases}
$$
 and $\gamma''(s)$ is trivial on $M'$. Hence, $q'_*[\gamma_0'] = q'_*[\gamma_2']= q'_*[\gamma_4']=  q'_*[\gamma_6']=  1$. 
From the relations \eqref{eqn:rel-fund-1}, and $[\gamma_5']^{-1}=[\gamma_0']^{-1}[\gamma_5][\gamma_0']$ we obtain $q'_*[\gamma_1']=  q'_*[\gamma_3'] = q'_*[\gamma_5']^2 = 1$. We finally show that, indeed, $q'_*[{\gamma}_5']=  1$.
We first consider the paths  
$$
\gamma_8, {\gamma}_9,{\gamma}_{10}  \colon [0,1]\to M, 
\quad \gamma_8(s)=[1/2+s/2,[0,0,0]],  \quad \gamma_{9}(s)=[1,[0,0,s/2]], \quad \gamma_{10}(s)=[1/2+s/2,[0,0,1/2]].
$$
There is a base-point preserving homotopy on $M$ 
from $\gamma_8 \cdot \gamma_{9} \cdot \gamma_{10}^{-1}$ to $\widetilde{\gamma}_5$.
 Since  $q\circ \gamma_8=q\circ \gamma_{10}$, this implies that $ (q\circ \gamma_8) \cdot  (q\circ \gamma_{9}) \cdot (q\circ \gamma_8)^{-1}$ is homotopic to ${\gamma}_5'$. We now observe that 
$$
\kappa \circ (q \circ \gamma_{9}) (s)=q[0,[0,0,s/2]]=q[1,[0,0,-s/2]]=q[1,[0,0,1/2-s/2]]=(q\circ \gamma_{9})^{-1}(s).
$$
The argument above ensures that the loop $q' \circ q \circ  \gamma_{9}$ is trivial on $X$. Hence, 
$$
q' \circ {\gamma}_{5}' \sim
(q' \circ q\circ\gamma_8) \cdot  (q' \circ q\circ \gamma_{9}) \cdot (q' \circ q\circ\gamma_8)^{-1} \sim (q' \circ q\circ\gamma_8) \cdot (q' \circ q\circ\gamma_8)^{-1} \sim 1.
$$

Therefore, $\pi_1(\widetilde{M})=\pi_1(X)=\{1\}$.
\end{proof}
\begin{remark}
    The resolution method by Joyce \cite[Theorem 2.2.3]{J2} also allows to desingularize the orbifold $X$ and to obtain a metric with holonomy $\Gtwo$ on the resolution. 
    We used the method in \cite{JK}, even if it involves harder analytic tools that are not necessary here, to take advantage of the author's previous knowledge about the coholomology algebra of the resolution (see \cite{LMM}).
\end{remark}

We denote by $\rho \colon \widetilde{M} \to X$ the projection map, and $E_j= \rho^{-1}(N_j)$. The resolution described in \cite{JK} occurs on the fibers of the normal bundle $\nu(N_j)$ of $N_j\subset X$, which are of the form $\C^2/\Z_2$. These singular fibers are replaced by their algebraic resolution,
$$
Y=\{ (w,\ell)\in \C^2\times \CP^1, \, w\in \ell \}/(w,\ell)\sim (-w,\ell).
$$
We describe this procedure in our set-up.  First, 
$N_j'=N_j^1/\iota$, where $\iota=\iota_2$ if $j=1,2$ of $\iota=\iota_1$ if $3\leq j \leq 10$. Indeed,
there is a $2:1$ cover between the normal bundles
$\nu(N_j^1) \to \nu(N_j')$ induced by the projection map $M \to M'$. The Deck transformation group of the cover is generated by the involution $\iota_*$. Both bundles have a complex structure $\mathrm{I}_j$ induced by the cross product with the vector field $\chi_j$ determined  by the harmonic form (see \cite[Remark 4.1]{JK}). That is,
 $\chi_j=\partial_{y_3}$ if $j=1,2$ or $\chi_j=\partial_{x_3}$ if $3\leq j \leq 10$, and $\mathrm{I}_j(X)$ is determined by the expression
$$
g(\mathrm{I}_j(X),Y)=\varphi(\chi_j,X,Y)
.
$$
The action of $\iota_*$ on $\nu(N_j^1)$ preserves $\mathrm{I}_j$ because $\iota_*$ fixes $\varphi$, $\chi_j$ and $g$, and the projection $\nu(N_j^1) \to \nu(N_j')$ is a complex homomorphism.
In addition, the bundle $\nu(N_j^1)$ is trivial. Since $\mathrm{I}_j(\partial_t)\in \la  \partial_{x_3} ,\partial_{y_3}  \ra$,  we pick a trivialization of the form $(\partial_t, I_j(\partial_t), X,I_j(X))$ where $X\in \la \partial_{x_1},\partial_{y_1}, \partial_{x_2} ,\partial_{y_2}  \ra \cap \nu(N_j^1)$. For instance,  if $j=1$ this is $(\partial_t, \partial_{x_3}, \partial_{x_1}+ \partial_{y_1}, -\partial_{x_2}- \partial_{y_2})$, and if $j=3$ this is $(\partial_t, -\partial_{y_3}, \partial_{y_1}, -\partial_{y_2})$. The orientation of $\nu(N_j^1)$ induced from $N_j$ and $M$ is determined by the form $(\chi_j \hook\varphi|_{\nu(N_j^1)})^2$, and the frame above is a positive basis. Therefore, the trivialization yields an isomorphism of oriented vector bundles $\nu(N_j^1)\cong N_j^1\times \C^2$, where $\C^2$ has the standard orientation.  For any $j$, under the isomorphism, the generator of the Deck group of $\nu(N_j^1) \to \nu(N_j')$ is $\iota_*(x,w_1,w_2)=(\iota(x),w_1,-w_2)$. 

The involution $\kappa$ induces the $\Z_2$-action on the fibers of $\nu(N_j')$ given by $v_x \to \kappa_*(v_x)=-v_x$ (see the proof of Lemma 2.8 in \cite{LMM}).
This  can be lifted to a map $\hat{\kappa} \colon N_j^1\times \CC^2 \to N_j^1\times \CC^2$ by means of the expression $\hat{\kappa}(x,w_1,w_2)= (x,-w_1,-w_2)$, which is induced by the involution on $M$ fixing that connected component. This is $\kappa_*$ if $j=1,3,4,5,6$, or $(\kappa\circ \iota_1)_*$ if $j=2$, or $(\kappa\circ \iota_2)_*$ if $j=7,8,9,10$. This action commutes with $\iota_*$. Thus $\nu(N_j)=\nu(N_j')/\kappa_*\cong (N_j^1 \times \CC^2)/\la \iota_*, \hat{\kappa} \ra \cong (N_j^1 \times \CC^2/\ZZ_2)/\iota_*$, and 
there is a $2:1$ cover $N_j^1\times (\C^2/\Z_2) \to (N_j^1\times \C^2/\ZZ_2)/ \iota_*$.

The resolution of $N_j^1\times \CC^2/\Z_2$ is $N_j^1\times Y$. 
The Deck transformation $\iota_*$ lifts to
 $N_j^1\times Y$ by means of the expression $\iota_*(x,[(w_1,w_2),[W_1:W_2]])=(\iota(x),[(w_1,-w_2), [W_1:-W_2]])$. The resolution $P(N_j)$ of $\nu(N_j)$ is the quotient $(N_j^1\times Y)/\iota_*$. Its exceptional divisor is
$Q(N_j)=(N_j^1\times \CP^1)/\iota_* \subset (N_j^1\times Y)/\iota_*$, which inherits the product orientation from $N_j^1\times \CP^1$, where $\CP^1$ is oriented by the Fubini-Study form. It turns out that $Q(N_j)$ is isomorphic to the trivial bundle because it is homotopic to it via
\begin{equation} \label{eqn:homotopy}
E_t=N_j^1\times \CP^1/(x,[W_1:W_2])\sim (\iota(x),[W_1:e^{\pi i t}W_2]), \qquad t \in [0,1].
\end{equation}
To obtain $\widetilde{M}$, the authors of \cite{JK} replace a neighborhood of $N_j$ in $X$ with a neighborhood of $Q(N_j)$ in $P(N_j)$. Therefore, $E_j=Q(N_j)$. The identification between these pieces is the composition of the natural projection $P(N_j)\to \nu(N_j)$, a rescaling of the fibers of $\nu(N_j)$, and a 
 modification of the exponential map that coincides with it over the zero section, is orientation-preserving, and is compatible with the involution (see \cite[Definition 3.2, equation (6.1)]{JK}).

\section{Nonformality of the constructed manifold} \label{sec:nonformality}

\subsection{Cohomology groups} \label{sec:cgroups}

We introduce some notations. We fix small tubular neighborhoods $O_j'\subset O_j''$ of $N_j'$ on $M'$, and denote by $\pi_j \colon O_j'' \to N_j'$ the nearest point projection. We define $O_j=q'(O_j)$, and $O_j^2=q(O_j'')$ where $q'\colon M' \to X$ is the quotient projection. To ease notations, we also set  $\pi_j \colon O_j^2 \to N_j$ the nearest point projection. In addition, we let $\mathrm{pr}_j= \pi_j \circ \rho \colon \rho^{-1}(O_j^2) \to N_j$. From the discussion in section \ref{sec:resolution} we deduce that it is a smooth map.

According to \cite[Proposition 6.1]{JK}, there are group isomorphisms 
\begin{equation} \label{eqn:cohom}
H^k(\widetilde{M})\cong H^k(X)\oplus \oplus_{j=1}^{10} H^{k-2}(N_j)\otimes \la \mathbf{x}_j\ra.
\end{equation}
The cohomology algebra $H^*(X)$ is computed from the complex $(\Omega(M')^{\kappa},d)$, and an average argument shows $H^*(X)=H^*(M')^{\kappa}$. There is an inclusion $H^k(X)\subset H^k(\widetilde{M})$ is determined by $\rho^*$ (see Lemma \ref{lem:smothness}). 
Each $\mathbf{x}_j$ has degree two, and is represented by the Thom class $[\tau_j]$ of the submanifold $E_j\subset \widetilde{M}$, oriented as in section \ref{sec:resolution}. Recall that the Thom class of $E_j$ is the generator of the compactly-supported cohomology group
$H_c^2(\rho^{-1}(O_j))$ that integrates to $1$ over each fiber of the normal bundle of $E_j$ inside $\widetilde{M}$ (see \cite[Proposition 6.18]{Bott-Tu}).  In addition, a closed form in $H^*(\widetilde{M})$ that maps to $[\alpha] \otimes \mathbf{x}_j\in H^*(N_j)\otimes  \mathbf{x}_j$ is $\mathrm{pr}_j^*(\alpha) \wedge \tau_j$.

%Each $\mathbf{x}_j$ has degree two, and the inclusion $H^k(X)\subset H^k(\widetilde{M})$ is determined by $\rho^*$. Here, the cohomology algebra $H^*(X)$ is computed from $(\Omega(M')^{\kappa},d)$, and an average argument shows $H^*(X)=H^*(M')^{\kappa}$.
%A representative of $\mathbf{x}_j$ is the Thom class $[\tau_j]$ of the submanifold $E_j\subset \widetilde{M}$, oriented as in section \ref{sec:resolution}.
%Recall that the Thom class of $E_j$ is the generator of the compactly-supported cohomology group
%$H_c^2(\rho^{-1}(O_j))$ that integrates to $1$ over each fiber of the normal bundle of $E_j$ inside $\widetilde{M}$ (see \cite[Proposition 6.18]{Bott-Tu}). % Here, the orientation of the fibers is determined by that of $E_j$ and $M$. 
%Finally, a closed form in $H^*(\widetilde{M})$ that maps to $[\alpha] \otimes \mathbf{x}_j\in H^*(N_j)\otimes  \mathbf{x}_j$ is $\mathrm{pr}_j^*(\alpha) \wedge \tau_j$.

Observe that given $\alpha \in \Omega^k(X)$, the restriction $\rho^*(\alpha)|_{\widetilde{M}-\cup_j E_j}$ is a smooth form because $\rho \colon \widetilde{M}-\cup_j E_j \to X-\cup_j N_j$ is a diffeomorphism. Next Lemma describes how to find a representative $\alpha_s$ of a cohomology class $[\alpha] \in H^*(X)$ whose pullback to $\widetilde{M}$ is smooth.

\begin{lem} \label{lem:smothness}
For any $1\leq j \leq 10$, if $\alpha \in \Omega^k(X)$ is closed on $O_j^2$, there is $\alpha_j \in \Omega^k(X)$ such that
\begin{enumerate}
    \item The forms $\alpha_j$ and $\alpha$ coincide on $X-O_j^2$. In addition, 
    $\alpha_j - \alpha=d\beta_j$ where $\beta_j \in \Omega^{k-1}(X)$ is supported on $O_j^2$.
    \item On $O_j$ we have $\alpha_j = \pi_j^*(\alpha_j|_{N_j})$. In particular, $\rho^*(\alpha_j)$ is smooth on $\rho^{-1}(O_j)$ because $\rho^*(\alpha_j)= \mathrm{pr}_j^*(\alpha_j|_{N_j})$ there.
\end{enumerate}
Therefore, given $[\alpha]\in H^k(X)$ there is a representative $\alpha_s \in \Omega^k(X)$ such that  $\rho^*(\alpha_s)$ is smooth and
 $\rho^*(\alpha_s)|_{\rho^{-1}(O_j)}=\mathrm{pr}_j^*(\alpha_s|_{N_j})$ for every $1\leq j \leq 10$.
\end{lem}
\begin{proof} 
The form $\beta=\pi_j^{*}(\alpha|_{N_j'})\in \Omega^2(O_j'')$ is $\kappa$-invariant because $\pi_j \circ \kappa= \pi_j$. In addition, it is closed and satisfies $\beta|_{N_j'}=\alpha|_{N_j'}$.  The Poincaré's Lemma allows to find a form $\theta_j$ on $O_j''$ such that $d\theta_j= (\beta - \alpha)|_{O_j''}$. 
Indeed, since $\beta$ and $\alpha$ are $\kappa$-invariant, we can assume that $\theta_j$ is. Let $h_j$ be a $\kappa$-invariant bump function which equals $0$ on $M'-O_j''$ and $1$ on $O_j'$. Then $h_j \theta_j\in \Omega^{k-1}(X)$ is supported on $O_j^2$ and 
$
\alpha_j= \alpha + d(h_j \theta_j)
$
satisfies the required properties.
\end{proof}

 \begin{remark}
In the proof of 
 \cite[Proposition 6.1]{JK}, and following the notation in our paper, they use 
that the bundles $Q(N_j)\to N_j$ are trivial. Their justification is that the normal bundle $\nu(N_j)$ is trivial as a real bundle (see their Remark 2.14). There could be a gap between these statements, because the chosen complex structure might not be constant in the trivialization suggested by the authors. The space of complex lines could then differ from fiber to fiber. This is why we showed in \eqref{eqn:homotopy} that $Q(N_j)$ is trivial.
 \end{remark}

\begin{proposition} \label{prop:cohomology-M}
The Betti numbers of  $\widetilde{M}$ are $(b_1,b_2,b_3)=(0,11,16)$.
The cohomology groups are:
\begin{align*}
H^2(\widetilde{M})=& \la dz_{1\bar{2}}+ dz_{\bar{1}2} \ra \oplus \la \mathbf{x}_1, \dots, \mathbf{x}_{10} \ra , \\
H^3(\widetilde{M})=&
dt\wedge \la idz_{1\bar{1}},i dz_{2\bar{2}}, i dz_{3\bar{3}}, i(dz_{1\bar{2}}- dz_{\bar{1}2}) \ra \oplus \la dz_{123}+ dz_{\bar{1}\bar{2}\bar{3}},  dz_{12\bar{3}}+ dz_{\bar{1}\bar{2}3} \ra \\
&\oplus \la dy_3 \otimes \mathbf{x}_1, dy_3 \otimes \mathbf{x}_2, dx_3 \otimes \mathbf{x}_3, \dots, dx_3 \otimes \mathbf{x}_{10} \ra.
\end{align*}
\end{proposition}
\begin{proof}
We first compute the cohomology groups of $M$. We observe,
$$
H^k(M)=H^k(\mathbb{T})^{F}=dt\wedge H^{k-1}(T^6)^f \oplus H^k(T^6)^f.
$$
It is then clear that $H^1(M)=\la dt \ra$; in addition, $\kappa^*(dt)=-dt$.
For the remaining degrees, we
analyze the map $f^* \colon H^k(T^6, \CC) \to H^k(T^6, \CC) $ and then we take real parts to find $H^k(T^6)^f$. That is, since $f$ is a holomorphic map, it preserves each of the subspaces $H^{p,q}(T^6)$ and given a $(p,q)$-form $\alpha$ with $p>q$, the map $f^*$ fixes 
$\alpha + \overline{\alpha}$  if and only if it fixes $\alpha$. Then, it suffices to study the action of $f^*$ on $H^{p,q}(T^6)$ with $p\geq q$.

The action of $f^*$ in terms of the basis $(dz_{13}, dz_{23}, dz_{12})$ of $H^{2,0}(T^6,\CC)$ is $\diag(-i,-i,-1)$. Consider the basis of $H^{1,1}(T^6,\C)$ given by $(dz_{1\bar{1}}, dz_{2\bar{2}},dz_{3\bar{3}}, dz_{1\bar{2}}, dz_{\bar{1}2},dz_{1\bar{3}}, dz_{2\bar{3}} ,dz_{\bar{1}3}, dz_{\bar{2}3})$, the action of $f^*$ is $\diag(1,1,1,1,1,-i,-i,i,i)$. Hence, 
$
H^2(T^6)^f=
\la idz_{1\bar{1}},idz_{2\bar{2}},idz_{3\bar{3}}, dz_{1\bar{2}}+ dz_{\bar{1}2}, i(dz_{1\bar{2}}- dz_{\bar{1}2}) \ra.
$
All these forms are invariant under $\iota_1$ and $\iota_2$. However, $\kappa$ acts as $\diag(-1,-1,-1,1,-1)$.

The map $f^*$ acts trivially on $H^{3,0}(T^6)$. Regarding $H^{2,1}(T^6,\CC)$, we consider the basis $(dz_{12\bar{3}},dz_{13\bar{2}}, dz_{23\bar{1}})$. Then, $f^*$ acts as $\diag(1,-1,-1)$. Thus,
 $$
 H^3(T^6)^f=
\la  dz_{123}+ dz_{\bar{1}\bar{2}\bar{3}}, i(dz_{123}- dz_{\bar{1}\bar{2}\bar{3}}),  dz_{12\bar{3}}+ dz_{\bar{1}\bar{2}{3}}, i(dz_{12\bar{3}}- dz_{\bar{1}\bar{2}{3}})\ra.
$$
These forms are fixed by $\iota_1$ and $\iota_2$. In addition, the action of $\kappa$ is $\diag(1,-1,1,-1)$. 
Finally, one deduces the result from our previous calculations, the equality $H^k(X)=H^k(M)^{\la \kappa, \iota_1,\iota_2 \ra}$ and equation \eqref{eqn:cohom}.
\end{proof}

\begin{remark}
As \cite[Figure 12.3]{Joyce2} shows, there is no holonomy $\Gtwo$ manifold with $(b_2,b_3)=(11,16)$ within the examples provided in section 12 of that book. The examples in \cite[section 8]{Kovalev} have $b_2 \leq 9$, and those in \cite[section 6]{KL} with $b_2=11$ have $b_3\geq 66$. 
\end{remark}

\subsection{Algebra structure} \label{sec:prod}

We now focus on the product under the isomorphism provided by equation \eqref{eqn:cohom}. This discussion is based on \cite[Proposition 6.4]{LMM} and uses the notations introduced in section \ref{sec:resolution}. Similar computations have been done in \cites{FFKM,Taimanov}.

We first claim that the Thom form of the normal bundle of $N_j^1\times \CP^1 \subset N_j^1 \times Y$ is 
 the extension of
\begin{equation} \label{eqn:Thom}
\tau_j= \frac{1}{\pi i} d(b(r^2)\partial \log(r^2))=\frac{2}{\pi i}b'(r^2)rdr \wedge \partial \log(r^2)+ \frac{b(r^2)}{\pi i}\overline{\partial}{\partial}\log r^2 .
\end{equation}
where $r$ is the radial coordinate on $\C^2$, and $b(s)$ is an increasing function that equals $-1$ on $s\leq \e^2$ and $0$ on $s\geq 4\e^2$. Observe that $\tau_j$ is smooth because the second term is proportional to the pullback of the Fubini-Study form on $\CP^1$.

 To verify this claim, consider a unit-length vector $(u_1,u_2) \in \C^2$ and the holomorphic embedding to a fiber of the tautological line bundle $F=\{ (w,[u_1:u_2]), \, w\in [u_1:u_2] \}$, given by $e \colon \C \to F$, $e(\lambda)=(\lambda u_1, \lambda u_2, [u_1:u_2])$. This induces an orientation-preserving embedding onto the fiber $F_p=\{p\}\times F/\Z_2$ of the bundle $N_j^1 \times Y \to N_j^1 \times \CP^1$ that we denote by $e_p \colon \C/\Z_2 \to F_p$.

We observe that 
$
e^*(\partial \log (r^2))= \partial(\log |\lambda|^2)
= \frac{d\lambda}{\lambda}
$, so that
$
e^*\tau_j= \frac{1}{\pi i} b'(|\lambda|^2) d\bar{\lambda}\wedge d \lambda
$. Since $e^*\tau_j$ vanishes when $|\lambda| \notin [\e,2\e]$ we have:
\begin{align*}
\int_{\C/\Z_2}{e_p^*(\tau_j)}=&
\frac{1}{2}\int_{\C}{e^*(\tau_j)}= 
\frac{1}{2\pi i}\int_{|\lambda| \in [\e,2\e]}{d\left( \dfrac{b(|\lambda|^2)d\lambda}{\lambda} \right)} \\
=& \frac{1}{2\pi i} \int_{|\lambda|=2\e}{\dfrac{b(|\lambda|^2)d\lambda}{\lambda} } - \frac{1}{2\pi i} \int_{|\lambda|=\e}{\dfrac{b(|\lambda|^2)d\lambda}{\lambda} } = \frac{1}{2\pi i} \int_{|\lambda|=\e}{ \frac{d\lambda}{\lambda}}= 1.
\end{align*}

This proves that $\tau_j$ integrates to $1$ over the fibers of the normal bundle, and therefore, it is the Thom class of that bundle. Note in addition, that the Thom class of the tautological line bundle over $\CP^1$ is $\frac{1}{2}\tau_j$.

The form $\tau_j$ is $\iota_*$ invariant, and hence, the Thom form of $Q(N_j)\subset P(N_j)$ is the pushforward of $\tau_j$. We also denote it by $\tau_j$. Since the second term in equation \eqref{eqn:Thom} is proportional to the pullback of the Fubini-Study form of $\CP^1$, the form $\tau_j^2$ vanishes on a neighborhood of $Q(N_j)$. Hence, it induces by pushforward a form $\rho_*(\tau_j^2)$ on $X$ with compact support around $N_j$. Let $\mathrm{Th}[N_j]$ be the Thom class of $N_j\subset X$, oriented by $\varphi|_{N_j}$, then
\begin{equation}\label{eqn:sq-thom}
[\rho_*(\tau_j^2)]= -2\mathrm{Th}[N_j] \in H^*(X),
\end{equation}
 because 
$
\int_{\CC^2/\Z_2}{\rho_*(\tau_j^2)}= \int_{Y}{\tau_j^2}= \int_{\CP^1} \tau_j= \frac{1}{\pi i} \int_{\CC \mathbb{P}^1}  \partial \overline{\partial} \log (r^2)=-2
$ (or alternatively, $\int_{Y}{\tau_j^2}= [\CP^1][\CP^1]=-2$).
In addition, when $j\neq k$ the supports of $\tau_j$ and $\tau_k$ are disjoint; hence $\tau_j \wedge \tau_k = 0$. 
Finally, we let $[\alpha]\in H^k(X)$, by Lemma \ref{lem:smothness}, we can assume that it satisfies $\alpha=\pi_j^*(\alpha|_{N_j})$ on $O_j$; then
$
\rho^*(\alpha)\wedge \tau_j= \mathrm{pr}_j^*(\alpha|_{N_j})\wedge \tau_j.
$
Under the isomorphism in equation \eqref{eqn:cohom}, our discussion yields the following product rules:
\begin{align}
\mathbf{x}_j^2=&-2\mathrm{Th}[N_j], \qquad \qquad \qquad \qquad  \mathbf{x}_{j}\cdot \mathbf{x}_{k}=0 \mbox{ if } j\neq k, \label{eqn:prod-1} \\
[\alpha] \cdot \mathbf{x}_{j}=& [\alpha|_{N_j}]\otimes \mathbf{x}_j, \mbox{ if } [\alpha] \in H^*(X). \label{eqn:prod-2}
\end{align}

Since the Thom class of $N_j$ (viewed in $H^*(X)$ rather than $H^*_c(O_j)$) coincides with its Poincaré Dual, we focus on 
 finding $PD[N_j]\in H^4(X)$.

\begin{lem} \label{lem:PDuals}
Denote $\b_1=i dt\wedge (dz_{123} - dz_{\bar{1}\bar{2}\bar{3}})$, and  $\b_2=idt \wedge (dz_{12\bar{3}} - dz_{\bar{1}\bar{2}3})$. Then 
$$
PD[N_1]=PD[N_2]=\beta_1+\beta_2, \qquad 
PD[N_3]=PD[N_j]=(\beta_1 - \beta_2)/2, \qquad 4 \leq j \leq 10.
$$
\end{lem}
\begin{proof}
Following the proof of Proposition \ref{prop:cohomology-M}, we obtain that $\b_1,\beta_2\in H^4(X)$. Let $\alpha \in H^3(X)$; then, $\alpha=dt\wedge\beta + \lambda_1(dz_{123} + dz_{\bar{1}\bar{2}\bar{3}})+ \lambda_2( dz_{12\bar{3}} + dz_{\bar{1}\bar{2}{3}})$, where $f^*\beta=\beta$ and $\kappa^*\beta=-\beta$.  Let $\mu_1,\mu_2\in \R$, since $M-\cup_{j,k} N_j^k \to X-\cup_{j}N_j$ is an $8:1$ cover we have
$$
\int_X {\alpha \wedge (\mu_1\beta_1 + \mu_2\beta_2)}=
\frac{1}{8} \int_{M} {-2i(\mu_1 \lambda_1- \mu_2 \lambda_2)dt\wedge dz_{1\bar{1}2\bar{2}3\bar{3}}} =
2(\mu_1 \lambda_1- \mu_2 \lambda_2).
$$

Using the coordinates described in \eqref{eqn:N1}, we compute $\alpha|_{N_1^k}= 4(\lambda_1-\lambda_2)dx_{12} \wedge dy_3$. Similarly, from equation \eqref{eqn:N2} we obtain $\alpha|_{N_2^k}= -4(\lambda_1-\lambda_2)dx_{12}  \wedge dy_3$. Taking into account the induced orientations \eqref{eqn:or-1}, \eqref{eqn:or-2} of $N_j^k$ we deduce
$\int_{N_j}\alpha=\frac{1}{2}\int_{N_j^1}{\alpha}=2(\lambda_1 - \lambda_2)$ for $j=1,2$. In addition, if $3\leq j \leq  6$ then $\alpha|_{N_j^k}=2(\lambda_1+\lambda_2)dx_{123}$ and if $7\leq j \leq 10$ then $\alpha|_{N_j^k}=-2(\lambda_1+\lambda_2) dy_{12} \wedge dx_3$. Of course, we used the coordinates in 
\eqref{eqn:N3-6}, \eqref{eqn:N7-10}.
Equations \eqref{eqn:or-3-6}, \eqref{eqn:or-7-10} ensure
$\int_{N_j}\alpha=\frac{1}{2}\int_{N_j^1}{\alpha}=(\lambda_1 + \lambda_2)$ for $j\geq 3$. The statement follows from these calculations.
\end{proof}

\subsection{A non-vanishing triple Massey product}

We begin by recalling the definition of a triple Massey product (see \cite[Definition 2.89]{FOT} or \cite[Definition 6.1]{OT}). In the sequel, given an $m$-form $\beta$, we denote $
\hat{\beta}=(-1)^{m}\beta.
$
Consider cohomology classes $[\alpha_1],[\alpha_2],[\alpha_3]\in H^*$, we say that their triple Massey product $\la [\alpha_1] ,[\alpha_2],[\alpha_3]\ra$ is well-defined if
$$
[\alpha_1][\alpha_2]=0, \qquad  [\alpha_2][\alpha_3]=0. 
$$
If these conditions are satisfied, we pick forms $\xi_{12}$, $\xi_{23}$ with $d\xi_{12}= {\alpha}_1 \wedge \alpha_2$ and  $d\xi_{23}= {\alpha}_2 \wedge \alpha_3$. Then, 
\begin{equation} \label{eqn:closed-Massey}
    y= \xi_{12} \wedge \alpha_3 - \hat{\alpha}_1 \wedge  {\xi}_{23}
\end{equation}
is a closed form. Let $\mathcal{I}$ be the ideal in $H^*$ generated by $[\alpha_1]$ and $[\alpha_3]$, and denote the projection map by $\pi \colon H^* \to H^*/\mathcal{I}$.
The triple Massey product is defined as $\la [\alpha_1] ,[\alpha_2],[\alpha_3]\ra= \pi[y]$. We say that it vanishes if $\pi[y]=0$, that is, if $[y]\in \mathcal{I}$.

It is well-known that all triple Massey products vanish on a formal manifold (see for instance \cite[Proposition 2.90]{FOT}). The goal of this section is to prove that the triple Massey product 
$\la [\tau_1 + \tau_2], [\tau_7+ \tau_3], [\tau_7-\tau_3] \ra$
on $H^*(\widetilde{M})$ does not vanish (see Theorem \ref{theo:non-formal}). Note that it is well-defined because $[\tau_7]^2-[\tau_3]^2=-2\rho^*(\mathrm{Th}[N_7]-\mathrm{Th}[N_3])=-2\rho^*(PD[N_7]-PD[N_3])=0$ by Lemma \ref{lem:PDuals}.

In order to find a primitive of $\tau_7^2-\tau_3^2$, we first construct a form $\beta \in \Omega^3(M)$ with $d\beta=\upsilon_3^1-\upsilon_7^1$, where $\upsilon_j^1$ is a form representing the Thom class of $N_j^1$, $j=3,7$. The construction uses a cobordism between $N_3^1$ and $N_7^1$.

\begin{proposition}\label{prop:prim-M}
There are small tubular neighborhoods $V_3^1 ,V_7^1$ of $N_3^1, N_7^1$ and forms
 $\beta \in \Omega^3(M)$, $\upsilon_3^1, \upsilon_7^1 \in \Omega^4(M)$, 
such that
\begin{enumerate}
\item The forms $\upsilon_3^1$, $\upsilon_7^1$ are closed and supported on $V_3^1$ and $V_7^1$ respectively. The cohomology class $[\upsilon_j^1]$ is the Thom class of $N_j^1$.
\item The form $\beta$ vanishes in a neighborhood of the level set $t=1/2$ and $d\beta=\upsilon_3^1-\upsilon_7^1$. In particular, it is closed on $M-(V_3^1\cup V_7^1)$ and it satisfies
 $[\beta|_{N_1^k}]=-\frac{1}{2}[\varphi|_{N_1^k}]$,  $[\beta|_{N_2^k}]=\frac{1}{2}[\varphi|_{N_2^k}]$.
\end{enumerate}
\end{proposition}
\begin{proof}
For convenience, we use $t\in [1/2,3/2]$ instead of $t\in [0,1]$. We first find a cobordism on $M$ between $N_3^1$ and $N_7^1$. We view points in $N_3^1$ at the level set $t=1/2$ and we consider the coordinates in \eqref{eqn:N3-6} and the orientation in \eqref{eqn:or-3-6}. Instead, we view $N_7^1$ inside the level set $t=3/2$. That is, since
$[1/2,[iy_1,iy_2,x_3- i/4]]= [3/2,[-y_1,-y_2,-x_3+i/4]]$, we describe:
\begin{equation} \label{eqn:N7-1-new}
N_7^1=\{ [3/2,[x_1,x_2,x_3 + i/4]], \quad x_1,x_2,x_3 \in [-1/2,1/2] \}.
\end{equation}
Its orientation is given by $\varphi|_{N_7^1}=dx_{123}$.
Let $\mathsf{f}\colon [1/2,3/2]\to [0,3/4]$ be a non-decreasing function that equals $0$ if $t\leq 9/8$, and $1/4$ if $t\geq 5/4$. Define the embedding $\Psi\colon [1/2, 3/2] \times N_3^1 \to M$,
\begin{equation} \label{eqn:cob}
\Psi (t,[1/2,x_1,x_2,x_3])=[t,[x_1,x_2,x_3 + i\mathsf{f}(t)]].
\end{equation}
Then $C=\mathrm{Im}(\Psi)$ is a manifold with boundary, $ N_3^1\cup N_7^1$. We trivialize and orient the bundle $TC$ using the frame $(-\partial_t- \mathsf{f}'(t)\partial_{y_3}, \partial_{x_1}, \partial_{x_2}, \partial_{x_3} )$. The normal bundle $\nu(C)$ is then oriented and trivialized by 
$(\partial_{y_1}, \partial_{y_2}, \partial_{y_3}- \mathsf{f}'(t)\partial_{t})$. The basis $(\partial_{x_1}, \partial_{x_2}, \partial_{x_3})$ orients $N_j^1$, and the outward pointing vector is $V_o=-\partial_t$ on $N_3^1$ and $V_o=\partial_t$ on $N_7^1$. Since a basis $(e_1,e_2,e_3)$ of $\partial C$ is positive if $(V_o,e_1,e_2,e_3)$ is positive on $C$, we have that
$\partial[{C}]=[N_3^1] - [N_7^1]$ as oriented manifolds.

If $j=3,7$, the bundle $\nu(N_j^1)$ is oriented by $(\partial_t,\partial_{y_1},\partial_{y_2},\partial_{y_3})$. Hence, $\nu(N_j^1)\cong \la \partial_t|_{N_j^1} \ra \oplus \nu(C)|_{N_j^1}$ as oriented bundles. The Thom form of $\nu(N_j^1)$ is then the product of the Thom forms of the bundles $\la \partial_t|_{N_j^1}  \ra$ and $\nu(C)|_{N_j^1}$ (see \cite[Proposition 6.19]{Bott-Tu}); this idea allows us to find $\beta$. 

Let  $\zeta \colon [-1/2,1/2]^3 \to \R$ be a positive radial function whose integral is $1$ and its support is contained in $B_{2\delta}(0)-B_{\delta}(0)$, for a small $\delta$. Consider a periodic extension of $\zeta$ to $\R^3/\Z^3$.
Set $h\colon [1/2,3/2]\to \R$ a bump function that equals $0$ on $[1/2,1/2+\e]\cup[3/2-\e,3/2]$, and $1$ on $[1/2+2\e,3/2-2\e]$. Then, 
\begin{equation} \label{eqn:beta}
\beta|_{[t,[x_1+iy_1,x_2+iy_2,x_3+iy_3]]}= h(t)\zeta(y_1,y_2,y_3-\mathsf{f}(t))dy_1\wedge dy_2 \wedge (dy_3-\mathsf{f}'(t)dt), \qquad t\in [1/2,3/2]
\end{equation}
induces a well-defined form on $M$ that vanishes around $t=1/2$. Indeed, around a small neighborhood of the interior of $C$, the expression
$
\zeta(y_1,y_2,y_3-\mathsf{f}(t)) dy_1\wedge dy_2 \wedge (dy_3-\mathsf{f}'(t)dt) 
$
defines a closed form whose pullback to $\nu(\mathrm{Int}(C))$ represents the Thom class of that oriented bundle. The support of the form
$$
d\beta= h'(t) \zeta(y_1,y_2,y_3-\mathsf{f}(t)) dt\wedge dy_1\wedge dy_2 \wedge dy_3, \qquad t\in [1/2,3/2]
$$
 is contained on a neighborhood of $N_3^1\cup N_7^1$, and we can write $d\beta= \upsilon_3^1-\upsilon_7^1$ where the support of $\upsilon_j^1$ is contained in the following neighborhood $V_j^1$ of $N_j^1$:
 \begin{align*}
V_3^1=&\{ [t,[x_1 +iy_1,x_2+iy_2,x_3+iy_3]], \quad |t-1/2|<2\e , \|(y_1,y_2,y_3) \|<2\delta  \}, \\
V_7^1=&\{ [t,[x_1 +iy_1,x_2+iy_2,x_3+iy_3]], \quad |t-3/2|<2\e , \|(y_1,y_2,y_3-1/4) \|<2\delta  \} .  
 \end{align*}
 
If $[t,[z]]\in V_3^1$ satisfies $1/2-2\e < t<1/2$ and  $\|(y_1,y_2,y_3) \|<2\delta$, from $\mathsf{f}(t+1)=1/4$, we obtain
 $d\beta|_{[t,[z]]}=d\beta|_{[t+1,[iz_1,iz_2,-z_3]]}= h'(t+1)\zeta(x_1,x_2,-y_3-1/4)dt\wedge dx_1\wedge dx_2 \wedge d(-y_3)$. This vanishes because $|-y_3-1/4|\geq 1/4-2\delta$ and $\zeta$ is supported around $0$.
 In the same way, $d\beta=0$ at points $[t,[z]]\in V_7^1$ with $3/2\leq t \leq 3/2+2\e$.
 
We now check that $[\upsilon_3^1]$ is the Thom class of $N_3^1$. First,
$\upsilon_3^1$ is closed because $d\beta$ is and the support of $\upsilon_3^1$ is disjoint to that of $\upsilon_7^1$. In addition, let $p \in N_3^1$ and let $F_p$ be the fiber of $V_3^1$ at $p$, oriented by $(\partial_t, \partial_{y_1}, \partial_{y_2}, \partial_{y_3})$. Since
$\upsilon_3^1$ is supported on the level sets  $1/2+\epsilon<t<1/2+2\epsilon$, where $\mathsf{f}=0$, we have
\begin{align*}
\int_{F_p}{\upsilon_3^1}=& \int_{t=1/2+\e}^{1/2+2\e}{h'(t)dt} \int_{\|(y_1,y_2,y_3) \|<2\delta}{\zeta(y_1,y_2,y_3) dy_1\wedge dy_2 \wedge dy_3}
= h(1/2+2\e)-h(1/2+\e)=1.
\end{align*}
Hence, $[\upsilon_3^1]$ is the Thom class of $N_3^1$. 
A similar computation shows that $[\upsilon_7^1]$ is the Thom class of $N_7^1$.

Of course, if $j=1,2$ and $k=1,2$ we have $d(\beta|_{N_j^k})=(\upsilon_3^1-\upsilon_7^1)|_{N_j^k}=0$. Let $\sigma_1=-1$ and $\sigma_2=1$; we now prove
$
\int_{N_j^k}{\beta|_{N_j^k}}=\sigma_j
$; this implies the last claim, because
$\int_{N_j^k}{\varphi}=2$.
For convenience, we rewrite the formulas of $N_j^k$ at the level set $t=1$:
\begin{align}
N_1^k=&\{[1,[x_1+ix_1,x_2+ix_2,-\e_k - iy_3]], \quad x_1,x_2,y_3 \in [-1/2,1/2]\}, \quad \e_1=0, \quad \e_2=1/2,\\
N_2^k=&\{[1,[-x_1+ix_1,-x_2+ix_2,-\delta_k - iy_3]], \quad x_1,x_2,y_3 \in [-1/2,1/2] \}, \quad \delta_1=-1/4, \quad \delta_2=1/4,
\end{align}
and these are oriented by $\varphi|_{N_1^k}=2 dx_{1}\wedge dx_2 \wedge dy_3$ and $\varphi|_{N_2^k}=-2 dx_{1}\wedge dx_2 \wedge dy_3$. Since $h(1)=1$ and $\mathsf{f}(1)=0$, we have:
\begin{align*}
\int_{N_1^k} \beta=& \int_{\|(x_1,x_2,-y_3)\|<2\delta} \zeta(x_1,x_2,-y_3) dx_1\wedge dx_2 \wedge d(-y_3)= -1,\\
\int_{N_2^k} \beta=& \int_{\|(x_1,x_2,-y_3)\|<2\delta} \zeta(x_1,x_2,-y_3) dx_1\wedge dx_2 \wedge d(-y_3)=1.
\end{align*}
Here we used that the map $(x_1,x_2,y_3) \to (x_1,x_2,-y_3)$ reverses the orientation, and the expressions for $\varphi|_{N_j^k}$. 
\end{proof}

\begin{remark}
Alternatively, for $j=3,7$ we can show $[\upsilon_j^1]=PD[N_j^1]$ by a direct computation. We focus on $j=3$.
According to the proof of Proposition \ref{prop:cohomology-M}, given $[\alpha] \in H^3(M)$, we can assume that $\alpha= dt\wedge \beta + \lambda_1 \re(dz_{123}) + \lambda_2\imag(dz_{123}) + \mu_1 \re(dz_{12\bar{3}}) + \mu_2\imag(dz_{12\bar{3}})$, with $f^*\beta=\beta$. 
We first observe
$$
\alpha \wedge d\beta = -(\lambda_1+\mu_1) h'(t)\zeta(y_1,y_2,y_3-\mathsf{f}(t)) dt\wedge dx_{123}\wedge dy_{123}, \qquad t\in [1/2,3/2]
$$
because $\imag(dz_{12{3}})\wedge dy_{123}= \imag(dz_{12\bar{3}})\wedge dy_{123}  =0$, and $\real(dz_{12{3}})\wedge dy_{123} = \real(dz_{12\bar{3}})\wedge dy_{123}=dx_{123}\wedge dy_{123}$. Since $\upsilon_3^1$ is the extension of $d\beta|_{V_3^1}$, and it is supported on the level sets  $1/2+\epsilon<t<1/2+2\epsilon$, where $\mathsf{f}=0$, we have
\begin{align*}
\int_{M} \alpha \wedge \upsilon_3^1=&\int_{V_3^1} \alpha \wedge d\beta 
=  - (\lambda_1+\mu_1) \int_{t=1/2+\e}^{t=1/2+2\e} h'(t)dt \int_{T^6} \zeta(y_1,y_2,y_3) dx_{123}\wedge dy_{123} = \lambda_1+\mu_1.
\end{align*}
Using the coordinates in \eqref{eqn:N3-6} and the orientation in \eqref{eqn:or-3-6}, we get $\alpha|_{N_3^1}=(\lambda_1+\mu_1) dx_{123}$, and $\int_{N_3^1} \alpha = \lambda_1+\mu_1$. This shows $PD[N_3^1]=[\upsilon_3^1]$.
\end{remark}

%%% Add a remark with a different proof that \beta=Poincaré dual when I change the version?
\begin{remark}
Following the notation of the proof, we provide a geometric justification of the equality $
\int_{N_j^k}{\beta|_{N_j^k}}=\sigma_j
$ for $j,k=1,2$.
 Observe that
 $N_j^k\cap C=\{p_j^k\}$, where $p_1^k=[1,[0,0,-\e_k]]$ and $p_2^k=[1,[0,0,-\delta_k ]]$. Pick 
 the basis $b_0=(-\partial_t, \partial_{x_1}, \partial_{x_2}, \partial_{x_3})$ of $TC|_{N_j^k}$, and consider the positively oriented  basis of $N_j^k$ viewed at the level set $t=1$
  given by $b_j=(-\sigma_j \partial_{x_1}+ \partial_{y_1}, -\sigma_j \partial_{x_2}+ \partial_{y_2}, \sigma_j \partial_{y_3})$ (note that $\varphi(-\sigma_j \partial_{x_1}+ \partial_{y_1}, -\sigma_j \partial_{x_2}+ \partial_{y_2}, \sigma_j \partial_{y_3})=2$). Then, 
$(b_0,b_1)$ is a negative basis for $T_{p_1^k}M$ and
$(b_0,b_2)$ is a positive basis for $T_{p_2^k}M$ because: 
\begin{align*}
&(-\partial_t) \wedge \partial_{x_1}\wedge \partial_{x_2} \wedge \partial_{x_3} \wedge (-\sigma_j \partial_{x_1}+ \partial_{y_1})\wedge ( -\sigma_j \partial_{x_2}+ \partial_{y_2})\wedge ( \sigma_j\partial_{y_3}) \\
&=
-\sigma_j \partial_t \wedge \partial_{x_1}\wedge \partial_{x_2} \wedge \partial_{x_3} \wedge \partial_{y_1}\wedge \partial_{y_2} \wedge \partial_{y_3}\\
&= \sigma_j \partial_t \wedge \partial_{x_1} \wedge \partial_{y_1} \wedge \partial_{x_2} \wedge \partial_{y_2} \wedge \partial_{x_3} \wedge \partial_{y_3}.
\end{align*}
Around $C$ and on the level sets where $h=1$, namely $1/2+2\e<t<3/2-2\e$, the form  $\beta=\zeta(y_1,y_2,y_3-\mathsf{f}(t)) dy_1\wedge dy_2 \wedge (dy_3-\mathsf{f}'(t)dt)$ is a representative of the Thom class of $\nu(C)$. Therefore,  
$[\beta|_{N_j^k}]$ is the Thom class of $p_j^k$ in $N_j^k$, oriented negatively if $j=1$ or positively if $j=2$. 
Hence, the integral of $\beta$ on $N_1^k$ (resp. $N_2^k$) equals $-1$ (resp. $1$).
\end{remark}

\begin{remark} A formula similar to \eqref{eqn:cob} allows to find
an oriented cobordism between $N_1^1$ and $N_2^1$.
\end{remark}

The projection of the cobordism $C$ obtained in the proof of Proposition \ref{prop:prim-M} to $M'$ determines a cobordism between $N_3'$ and $N_7'$ that intersects $N_1'$ negatively and $N_2'$ positively. The average of the form $\alpha$ provides a primitive for the difference of the Thom forms of $N_3'$ and $N_7'$. We make this precise in the following result.

\begin{lem} \label{lem:prim-M'}
There are small tubular neighborhoods $V_3',V_7'$ of $N_3',N_7'$ on $M'$ and $\kappa$-invariant forms
 $\alpha \in \Omega^3(M')$, $\upsilon_3, \upsilon_7 \in \Omega^4(M')$, 
such that
\begin{enumerate}
\item The forms $\upsilon_3$, $\upsilon_7$ are closed and supported on $V_3'$ and $V_7'$ respectively. The cohomology class $[\upsilon_j]$ is the Thom class of $N_j'$.
\item The pushforwards of $2\upsilon_3$ and $2\upsilon_7$ to $X$ represent the Thom class of $N_3$ and $N_7$ on $X$ respectively.
\item The form $\alpha$ vanishes in a neighborhood of the level set $t=1/2$ and $d\alpha=\upsilon_3-\upsilon_7$. In particular, $\alpha$ is closed on $M'-(V_3'\cup V_7')$.
\item There are tubular neighborhoods $V_1'$ and $V_2'$ of $N_1'$ and $N_2'$ such that
$
\alpha|_{V_1'}=-\pi_1^*(\varphi|_{N_1'}), 
$
and 
$
\alpha|_{V_2'}=\pi_2^*(\varphi|_{N_2'}), 
$
where $\pi_j \colon V_j' \to N_j'$ denotes the nearest point projection for $j=1,2$.
\end{enumerate}
\end{lem}
\begin{proof}
For $j=3,7$ we denote $V_j^2=\iota_2(V_j^1)\subset M$
and $K=\la \iota_1,\iota_2,\kappa \ra=\Z_2^3$.
We let $K_3^1=\{\mathrm{Id}, \iota_1, \kappa, \kappa \circ \iota_1\}$ and $K_7^1=\{\mathrm{Id}, \iota_1, \kappa \circ \iota_2, \kappa \circ \iota_1 \circ \iota_2\}$.
We first observe that every element of $K_j^1$ maps $N_j^1$ onto itself, and since the elements in $K_j^1$ are isometries, they preserve $V_j^1$ (and its fiber bundle structure). In addition, the elements of 
$K_j^2=K-K_j^1$
swap $N_j^1$ with $N_j^2$ and therefore, $V_j^1$ with $V_j^2$. 
Of course, if $\gamma\in K$ maps $N_j^1$ to $N_j^k$ with $k=1$ or $k=2$, then the restrictions $\gamma\colon N_j^1 \to N_j^k$ and $\gamma_* \colon {\nu(N_j^1)}\to {\nu(N_j^2)}$ preserve the fixed orientations because the submanifolds $N_j^k$ are oriented by
 $\varphi|_{N_j^k}$ and all the elements in $K$ fix $\varphi$ and preserve the orientation of $M$. Hence, if
$j=3,7$, the form $\upsilon_j'=\frac{1}{4}\sum_{\gamma\in K_j^1}\gamma^*\upsilon_j^1$ is another representative of the Thom class of $N_j^1$ in $M$, and $\upsilon_j''=\frac{1}{4}\sum_{\gamma\in K_j^2}\gamma^*\upsilon_j^1$ represents the Thom class of $N_j^2$.

The pushforward $\upsilon_j$ of $\upsilon_j'+\upsilon_j''=\frac{1}{4} \sum_{\gamma \in K} \gamma^* \upsilon_j^1$ to $M'$ a is $\kappa$-invariant form that represents the Thom class of $N_j'\subset M'$. It is supported in the projection of $V_j^1\cup V_j^2$ to $M'$; that we denote by $V_j'$.
We observe that (the pushforward of) $2\upsilon_j$ induces a representative of the Thom class of $N_j$ on $X$. The reason is that
 the fiber of the normal bundle at a point $p\in N_j'$ is $F_p\cong \C^2$, and on $X$ this is $F_p/\Z_2$; hence,
$
\int_{F_p/\Z_2}{2\upsilon_j}=\frac{1}{2}\int_{F_p}{2\upsilon_j}=1.
$
Consider the $K$-invariant form $\alpha'=\frac{1}{4}\sum_{\g\in K} \g^*\beta$. On $M$ we have $$
d\alpha'=\frac{1}{4}\sum_{\g\in K} \g^*\upsilon_3^1 - \frac{1}{4}\sum_{\g\in K} \g^*\upsilon_7^1
=(\upsilon_3'+\upsilon_3'')-(\upsilon_7'+\upsilon_7'').
$$
Hence, the pushforward of $\alpha'$ (that we keep denoting by $\alpha'$) satisfies $d\alpha'=\upsilon_3-\upsilon_7$. Since $\beta$ vanishes in a neighborhood of the level set $t=1/2$ and the elements of $K$ preserve that level set, we have that $\alpha'$ vanishes in a neighborhood of $t=1/2$.

Similarly to Lemma \ref{lem:smothness}, we modify $\alpha'$ around the level set $t=0$ to obtain $\alpha$ satisfying all these conditions. We again set $\sigma_1=-1$ and $\sigma_2=1$. For $j=1,2$, if $\gamma \in K$ then $\gamma$ maps $N_j^1$ onto itself, or it swaps $N_j^1$ and $N_j^2$. Since $\gamma$ preserves the orientations in both cases, we have 
$
\int_{N_j^k} \gamma^*\beta= \int_{\gamma(N_j^k)} \beta= {\sigma_j}= \frac{\sigma_j}{2}\int_{N_j^k} \varphi.  
$
Hence, $[(\gamma^*\beta)|_{N_j^k}]=\frac{\sigma_j}{2}[\varphi|_{N_j^k}]$ and  therefore $[\alpha'|_{N_j^k}]=\sigma_j[\varphi|_{N_j^k}]$. Thus,  $[\alpha'|_{N_j'}]=\sigma_j[\varphi|_{N_j'}]$ on $M'$.

By the Poincaré's Lemma given a tubular neighborhood $W_j'$ of $N_j'$, there is $\theta_j\in \Omega^2(W_j')$
such that $d\theta_j= \sigma_j \pi_j^*(\varphi|_{N_j'}) - \alpha'|_{W_j'}$. 
Indeed, we can assume that $\theta_j$ is $\kappa$-invariant as both $\pi_j^*(\varphi|_{N_j'})$ and $\alpha'$ are. Consider $V_1',V_2'$ smaller tubular neighborhoods with $\overline{V}_j'\subset W_j'$, and $\kappa$-invariant bump functions $h_j$ that equal $1$ on $V_j'$ and $0$ on $M'-W_j'$. The form
$$
\alpha= \alpha' + d(h_1\theta_1) + d(h_2\theta_2).
$$
satisfies all the stated conditions.
\end{proof}

The equality \eqref{eqn:sq-thom}, and Lemma \ref{lem:prim-M'} enable us to 
find a primitive of $\tau_7^2-\tau_3^2$ in Lemma \ref{lem:prim-wtM}. To state it, we introduce some notation. For $j=1,2,3,7$, we define $V_j$ as the projection of $V_j'\subset M'$ onto $X$ and we let $U_j=\rho^{-1}(V_j)$. Recall that at the end of section \ref{sec:manifold} we fixed a neighborhood $O_j$ of $N_j\subset X$ and we assumed that $\tau_j^2$ is supported on $\rho^{-1}(O_j)$. Observe that it is not
restrictive to assume that $O_j \subset U_j$, and we will assume that.

\begin{lem} \label{lem:prim-wtM}
There is a form $\widetilde{\alpha} \in \Omega^3(\widetilde{M})$ such that $d(\widetilde{\alpha})=\tau_7^2-\tau_3^2$.
In particular, it is closed on $U_1\cup U_2$. In addition,
$$
\widetilde{\alpha} \wedge \tau_1= -4\mathrm{pr}^*_1(\varphi|_{N_1}) \wedge \tau_1, \qquad
\widetilde{\alpha} \wedge \tau_2= 4\mathrm{pr}^*_2(\varphi|_{N_2}) \wedge \tau_2,
$$
where $\mathrm{pr}_j=\pi_j \circ \rho \colon \rho^{-1}(V_j') \to N_j$.
\end{lem}
\begin{proof}
Let $j=3,7$. We first observe that $\rho^*(\upsilon_j)$ is smooth because it vanishes on a small neighborhood of $N_j$ as $\alpha$ vanishes on a small neighborhood of $t=1/2$ and $d\alpha=\upsilon_3 -\upsilon_7$. Since $\tau_j^2$ vanishes on a neighborhood of $E_j$, the pushforward $\rho_*(\tau_j^2)$ represents an element in $H^4_c(V_j)$. 
This group is generated by the Thom class $\mathrm{Th}[N_j]_c$, which is represented by $2\upsilon_j$ according to Lemma \ref{lem:prim-M'}. In fact, equation \eqref{eqn:sq-thom} shows
$[\rho_*(\tau_j^2)]_c=-2\mathrm{Th}[N_j]_c=-2[2\upsilon_j]_c$. Hence, there are forms $\eta_3,\eta_7\in \Omega(X)$ with compact support on $V_j$ such that
$d\eta_j= \rho_*(\tau_j^2) +4\upsilon_j$.
 
The strategy outlined in Lemma \ref{lem:smothness} allows to modify $\eta_j$ around $N_j$ and find a form $\eta_j'$ supported on $V_j$
so that its pullback $\rho^*(\eta_j')$ is a smooth form on $\widetilde{M}$ and $d\eta'_j=\rho_*(\tau_j^2) +4\upsilon_j$. This is possible because on a small tubular neighborhood $O_j^3$ of $N_j$ both $\rho_*(\tau_j^2)$ and $\upsilon_j$ vanish, so
$d\eta_j|_{O_j^3}= (\rho_*(\tau_j^2) +4\upsilon_j)|_{O_j^3}=0$.  We now argue that $\rho^*(\alpha)$ is smooth. Since $\alpha=0$ on a neighborhood of the level set $t=1/2$, it is smooth on $M-(U_1\cup U_2)$. In addition, on $M'$ we had $\alpha|_{V_j'}=\sigma_j\pi_j^*(\varphi|_{N_j'})$, where  $\sigma_1=-1$ and $\sigma_2=1$. The same identity holds for their pushforwards to $X$, so that $\rho^*(\alpha)|_{U_j}=\rho^*(\alpha|_{V_j})= \sigma_j\mathrm{pr}_j^*(\varphi|_{N_j})$. Hence, it is smooth.

Finally, we define $\widetilde{\alpha}= \rho^*(\eta_7' +4 \alpha - \eta_3')$. Then $d\widetilde{\alpha}= \tau_7^2 + 4\rho^*\upsilon_7 + 4\rho^*(\upsilon_3 - \upsilon_7) - (\tau_3^2 + 4\rho^*(\upsilon_3))= \tau_7^2-\tau_3^2 $. Given $j=1,2$,  the support of $\tau_j$ is contained in $O_j\subset U_j$ and both $\rho^*(\eta_3')$ and $\rho^*(\eta_7')$ vanish there. Hence, 
$
\widetilde{\alpha} \wedge \tau_j= 4 \rho^*(\alpha)\wedge\tau_j= \sigma_j 4 \mathrm{pr}_j^*(\varphi|_{N_j}) .
$ 
\end{proof}

We finish by computing the triple Massey product $\la [\tau_1 + \tau_2], [\tau_7+ \tau_3], [\tau_7-\tau_3] \ra$.

\begin{theorem} \label {theo:non-formal}
The triple Massey product $\la [\tau_1 + \tau_2], [\tau_7+ \tau_3], [\tau_7-\tau_3] \ra$ is not trivial.
Therefore, $\widetilde{M}$ is non-formal.
\end{theorem}
\begin{proof}
Note that $(\tau_1 + \tau_2)\wedge (\tau_7+ \tau_3) =0$ and $(\tau_7+ \tau_3)\wedge (\tau_7-\tau_3)=\tau_7^2-\tau_3^2=d\widetilde{\alpha}$. The cohomology class determined by expression \eqref{eqn:closed-Massey} is
$$
[y]= [\widetilde{\alpha}\wedge (\tau_1 + \tau_2)],
$$
and the triple Massey product is trivial if and only if $y\in \mathcal{I}$, where
$\mathcal{I}$ is the ideal generated by $\tau_1 + \tau_2$ and $\tau_7-\tau_3$. We prove by direct computation that $y\notin \mathcal{I}$. 

First, by Lemma \ref{lem:prim-wtM} we know
$$
[y]= -4[\mathrm{pr}_1^*(\varphi|_{N_1}) \wedge \tau_1] + 4 [\mathrm{pr}_2^*(\varphi|_{N_2}) \wedge \tau_2].
$$
Under the isomorphism in equation \eqref{eqn:cohom}, this class is $y'=-4[\varphi|_{N_1}] \otimes \mathbf{x}_1 + 4[\varphi|_{N_2}]\otimes \mathbf{x}_2$.
We now compute the image  $\mathcal{I}'$ of the degree-5 part of the ideal $\mathcal{I}$ under the isomorphism,
namely
$$
H^3(X)\cdot \la \mathbf{x}_1+\mathbf{x}_2 \ra \oplus
H^3(X) \cdot \la \mathbf{x}_7-\mathbf{x}_3 \ra
\oplus (\la dy_3 \otimes \mathbf{x}_1, dy_3 \otimes \mathbf{x}_2, dx_3 \otimes \mathbf{x}_3, \dots, dx_3 \otimes \mathbf{x}_{10} \ra)\cdot (\la \mathbf{x}_1+\mathbf{x}_2 ,\mathbf{x}_7-\mathbf{x}_3 \ra) .
$$

The proof of Lemma \ref{lem:PDuals} ensures that 
 $\int_{N_1} \xi= \int_{N_2} \xi$ and  $\int_{N_3} \xi= \int_{N_7} \xi$ for every $[\xi] \in H^3(X)$.
 Fix $[\xi] \in H^3(X)$ and set $\lambda_1,\lambda_3\in \R$ so that $[\xi|_{N_j}]=\lambda_j [\varphi|_{N_j}]$ with $j=1,3$. Then, 
 $\int_{N_2}\xi=\int_{N_1}\xi= \lambda_1 \int_{N_1} \varphi= \lambda_1 \int_{N_2} \varphi$ and similarly $\int_{N_7}\xi= \lambda_3 \int_{N_7} \varphi$, that is, $[\xi|_{N_2}]=\lambda_1 [\varphi|_{N_2}]$, and  $[\xi|_{N_7}]=\lambda_3 [\varphi|_{N_7}]$. Hence, using equation \eqref{eqn:prod-2} we conclude, 
$$
[\xi] \cdot (\mathbf{x}_1+\mathbf{x}_2)= \lambda_1 ([\varphi|_{N_1}]\otimes \mathbf{x}_1+[\varphi|_{N_2}]\otimes \mathbf{x}_2), \qquad 
[\xi] \cdot (\mathbf{x}_7-\mathbf{x}_3)= \lambda_3 ( [\varphi|_{N_7}]\otimes \mathbf{x}_7-[\varphi|_{N_3}]\otimes \mathbf{x}_3).
$$
Thus,
$
H^3(X)\cdot \la \mathbf{x}_1+\mathbf{x}_2 \ra= \la [\varphi|_{N_1}]\otimes \mathbf{x}_1+[\varphi|_{N_2}]\otimes \mathbf{x}_2  \ra$, and  
$H^3(X)\cdot \la \mathbf{x}_7-\mathbf{x}_3 \ra= \la [\varphi|_{N_7}]\otimes \mathbf{x}_7-[\varphi|_{N_3}]\otimes \mathbf{x}_3  \ra .
$

In addition, 
$(\la dy_3 \otimes \mathbf{x}_1, dy_3 \otimes \mathbf{x}_2, dx_3 \otimes \mathbf{x}_3, \dots, dx_3 \otimes \mathbf{x}_{10} \ra)\cdot (\la \mathbf{x}_1+\mathbf{x}_2 ,\mathbf{x}_7-\mathbf{x}_3\ra)  \subset H^5(X)$ because
 $\mathbf{x}_j\mathbf{x}_k \in H^4(X)$ as a consequence of equation \eqref{eqn:prod-1}. For instance, $(dy_3\otimes \mathbf{x}_1) \cdot \mathbf{x}_1$ is represented by $\mathrm{pr}_1^*(dx_3) \wedge \tau_1^2$, and it induces a cohomology class on $X$ by pushforward as it is supported around $O_1$, and vanishes on a neighborhood of $E_1$. However, 
 $$
 y'=-4[\varphi|_{N_1}] \otimes \mathbf{x}_1 +4[\varphi|_{N_2}]\otimes \mathbf{x}_2 \notin H^5(X) \oplus \la [\varphi|_{N_1}]\otimes \mathbf{x}_1+[\varphi|_{N_2}]\otimes \mathbf{x}_2, [\varphi|_{N_7}]\otimes \mathbf{x}_7-[\varphi|_{N_3}]\otimes \mathbf{x}_3 \ra \supset \mathcal{I}',
 $$
 so $[y]\notin \mathcal{I}$.
\end{proof}

\end{document}